\documentclass[10pt,twoside,english]{article}

\usepackage[toc,page]{appendix}
\usepackage{latexsym}
\usepackage{amsmath,amsthm,mathrsfs}
\usepackage{amsfonts}
\usepackage{amssymb}
\usepackage{color}

\def\leq{\leqslant}
\def\geq{\geqslant}
\def\N{\mathbb{N}}
\def\R{\mathbb{R}}

\def\p{\mathbb{P}}

\def\omg{\ensuremath{\underline{\omega}}}
\def\eps{\varepsilon}

\def\qand{\quad\text{and}\quad}

\def\A{\ensuremath{A^{(q)}}}
\def\e{{\ensuremath{\rm e}}}

\def\I{\ensuremath{{\bf 1}}}
\usepackage[T2A,T1]{fontenc}
\DeclareSymbolFont{cyrillic}{T2A}{cmr}{m}{n}
\DeclareMathSymbol{\D}{\mathalpha}{cyrillic}{196}

\newtheorem{proposition}{Proposition}[section]
\newtheorem{example}{Example}[section]
\newtheorem{lemma}[proposition]{Lemma}
\newtheorem{definition}[proposition]{Definition}

\newtheorem{theorem}[proposition]{Theorem}
\newtheorem{remark} [proposition]{Remark}

\theoremstyle{definition}
\newtheorem*{condition}{Condition}

\begin{document}
\bibliographystyle{plain}

\title {Extreme Value Theory for Piecewise Contracting Maps with Randomly Applied Stochastic Perturbations}
\author{Davide Faranda \thanks{Laboratoire SPHYNX, Service de Physique de l'Etat Condens\'e, DSM, CEA Saclay, CNRS URA 2464, 91191 Gif-sur-Yvette, France. e-mail: $<$davide.faranda@cea.fr$>$.} \and Jorge Milhazes Freitas \thanks {Centro de Matem\'{a}tica \& Faculdade de Ci\^encias da Universidade do Porto, Rua do
Campo Alegre 687, 4169-007 Porto, Portugal. e-mail:$<$jmfreita@fc.up.pt$>$.}
\and Pierre Guiraud \thanks{CIMFAV, Facultad de Ingenier\'ia, Universidad de Valpara\'iso, Valpara\'iso, Chile. e-mail: $<$pierre.guiraud@uv.cl$>$.}
\and Sandro Vaienti
\thanks{Aix Marseille Universit\'e, CNRS, CPT, UMR 7332, 13288 Marseille, France and
Universit\'e de Toulon, CNRS, CPT, UMR 7332, 83957 La Garde, France.
e-mail:$<$vaienti@cpt.univ-mrs.fr$>$.}}
\maketitle

\begin{abstract} We consider globally invertible and piecewise contracting maps in higher dimensions and we perturb them with a particular kind of noise introduced by Lasota and Mackey. We got random transformations which are given by a stationary process: in this framework we develop an extreme value theory for a few classes of observables and we show how to get the (usual) limiting distributions together with an extremal index depending on the strength of the noise.
\end{abstract}

\section{Introduction}
In a previous note \cite{FFGV} we announced a few results on the application of Extreme Value Theory (EVT) to one-dimensional piecewise contracting maps (PCM)  whenever the latter are perturbed with a particular kind of noise. This noise was introduced in the book by Lasota and Mackey \cite{LM} and it was called {\em randomly applied stochastic perturbation} (RASP). It turns out that it is particularly useful to study PCM especially because it allows  very efficient representations of the transfer operator (Perron-Frobenius) and of the stationary measures. We now generalize to higher dimensional PCM our previous results by proving the existence of extreme value laws (EVL) for some kinds of observables associated to rare visits in the neighborhood of a given point. As a matter of fact, we first consider  globally invertible maps defined on compact subsets of $\mathbb{R}^n$ and which are diffeomorphisms but on some singular (Borel) sets of zero Lebesgue measure, so they are not necessarily contractions. The RASP belongs to the class of {\em random transformations} (see next section), which naturally define a Markov chain with associated a unique absolutely continuous  stationary measure (see Th. 10.4.2 in \cite{LM}). By taking the direct product of this stationary measure with the probability distribution  of the    concatenated maps, we get the (annealed) probability $\mathbb{P}$ which will govern the statistical properties of our system. In particular the stochastic process given by the observation along a random orbit will be stationary according to $\mathbb{P}.$  We will first prove a general result about decay of correlations which will not require strong assumptions on the functional spaces of the observables (Proposition \ref{Prop:Correlations}). We will therefore use this result to recover some sort of asymptotic independence in order to prove convergence of the distribution of the maxima toward the Gumbel's law. This will be accomplished in two other steps: first we will get an explicit expression of the affine coefficients defining the sequence of levels $u_n$ needed to rescale the distribution of the maximum to avoid degenerate limits (Propositions \ref{PT},\ref{ANBN},\ref{FIN}); then we will control the short returns in Proposition \ref{EVLLN}. Whenever we will consider rare visits in the neighborhood of a periodic cycle, we will prove the existence of an extremal index less than $1$, and in particular equal to the magnitude of the noise. \\We would like to stress that we already investigated the EVT for random transformations obtained by perturbing with additive noise; we applied it to uniformly expanding maps \cite{AFV} and even to rotations \cite{FFLTV}. In both cases we got the existence of an EVL around any point and with extremal index equal to $1$. It seems that RASP is much more suitable for piecewise contractions and does not smoothen out  completely the periodicity features of the unperturbed map, as this is reflected in the persistency of an extremal index less than one. A few numerical simulations seem to suggest however the existence of an EVL even when we perturb PCM  with additive noise \cite{FFGV}.\\

We now briefly recall the main statements of the EVT.  Let us  suppose that $(Y_n)_{n\in\N}$ is a sequence of real-valued random variables defined on the probability space $(\Psi, \mathbb{P}).$ We will be interested in the distribution of the maximum $M_n:=\max\{Y_0, Y_1, \dots, Y_{n-1}\}$ when $n\rightarrow \infty.$ It is well known that the limiting distribution is degenerate unless one proceed to a suitable re-scaling of the levels of exceedances. The precise formulation is the following:
 we have an Extreme Value Law  for $(M_n)_{n\in\N}$ if there is a non-degenerate distribution function
  $H:\R\to[0,1]$ with $H(0)=0$ and,  for every $\tau>0$, there exists a sequence of levels $(u_n(\tau))_{n\in\N}$ such that
\begin{equation}
\label{eq:un}
\lim_{n\to\infty}n\mathbb{P}(Y_0>u_n)\to \tau,
\end{equation}
and for which the following holds:
\[
\lim_{n\to\infty}\mathbb{P}(M_n \le u_n)\rightarrow 1-H(\tau).
\]

The motivation for using a normalizing sequence $(u_n)_{n\in\N}$ satisfying \eqref{eq:un} comes from the case when $(Y_n)_{n\in\N}$ are independent and identically distributed. In this i.i.d.\ setting, it is clear that $\mathbb{P}(M_n\leq u)= (F(u))^n$, being $F(u)$ the cumulative distribution function for the variable $u$. Hence, condition \eqref{eq:un} implies that
\[
\mathbb{P}(M_n\leq u_n)= (1-\mathbb{P}(Y_0>u_n))^n\sim\left(1-\frac\tau n\right)^n\to e^{-\tau},
\]
as $n\to\infty$.  Note that in this case $H(\tau)=1-e^{-\tau}$ is the standard exponential distribution function. Let us now choose the sequence $u_n=u_n(y)$ as the one parameter family $u_n=y/a_n+b_n$, where $y\in \mathbb{R}$ and $a_n>0,$ for all $n\in \mathbb{N}$. Whenever the variables $Y_i$ are i.i.d. and  for some constants $a_n>0$, $b_n\in\R$, we have $\mathbb{P}(a_n(M_n-b_n)\leq y)\rightarrow G(y)$, where the convergence occurs at continuity points of $G$, and $G$ is non-degenerate, then $G_n$ will converge to one of the three EVLs: Gumbel, Fr\'echet or Weibull.  The law obtained depends on the the common distribution of the random variables.

When $Y_0,Y_1,Y_2,\ldots$ are not independent, the standard exponential law still applies under some conditions on the dependence structure. These conditions will be stated in detail later and  they are usually designated with $D'$ and $D_2$; when they
 hold for  $(Y_n)_{n\in\N}$ then there exists an extreme value law for $M_n$ and $H(\tau)=1-e^{-\tau},$ see Theorem 1 in \cite{FF}. We want to stress that these two conditions alone do not imply the existence of an extreme value law;  they require, even to be checked, that the limit (\ref{eq:un}) holds. It turns out that for the kind of observables we are going to introduce, and which are related to the local properties of the invariant measure, the limit (\ref{eq:un}) is difficult to prove when the invariant measure is not absolutely continuous, since one needs the exact asymptotic behavior of that measure on small balls. Instead it turns out that whenever the systems is randomly perturbed and the invariant measure is absolutely continuous, the existence of the limit (\ref{eq:un}) can be insured with general arguments and in the majority of the cases (see below for a precise meaning), it can be expressed  by a closed formula in terms of the strength of the noise,  Proposition \ref{ANBN} and \ref{FIN} below.  Moreover that formula could be used in a reversed way: since the sequence $(u_n)_{n\in\N}$ is now uniquely determined for any $n$, a numerical sampling for $(u_n)_{n\in\N}$ which provides convergence to the extreme value law, will bring information on the local geometrical  properties of the stationary measure: this approach was successfully used, for instance, in \cite{FV1, FV2}. \\

\section{Random dynamical systems: RASP perturbation}

Let us consider  a sequence of i.i.d. random variables $(W_k)_{k\in \N}$ with values $(\omega_k)_{k\in \N}$ in a space $\Omega_{\varepsilon}$ and with common probability distribution $\theta_{\varepsilon}$. Let $X\subset\R^D$ be a compact set equipped with the Lebesgue measure $m$ defined on the Borel $\sigma$-algebra, and $(f_\omega)_{\omega\in \Omega_{\varepsilon}}$ a family of measurable transformations such that $f_{\omega}:X\to X$ for all $\omega\in \Omega_{\varepsilon}$\footnote{In the following when we will refer to a dynamical system $(X,f,\mu)$ we will mean that $f$ is defined on $X$ and preserves the Borel probability measure $\mu;$ if we will  write $(X,f)$, this will simply  correspond  to the action of $f$ on $X.$ }. Given a point $x\in X$ and a realization $\underline{\omega}= (\omega_1,\omega_2,\dots)\in \Omega_{\varepsilon}^\N$ of the stochastic process $(W_k)_{k\in \N}$, we define the random orbit of $x$ as the sequence $(f^n_{\underline{\omega}}(x))_{n\in\N}$, where
\[
f^0_{\underline{\omega}}(x)=x\quad\text{and}\quad f^n_{\underline{\omega}}(x)=f_{\omega_n}\circ f_{\omega_{n-1}}\circ\cdots\circ f_{\omega_1}(x)\qquad\forall n\geq 1.
\]
The transformations $f_{\omega}$ will be considered as stochastic perturbations of a deterministic map $f$, in the sense that they will be taken in a suitable neighborhood of $f$ whose {\em size} will be determined by the value of $\eps$, see below.  We could therefore  define a Markov process on $X$ with transition function
\begin{equation}\label{gre}
 L_{\varepsilon} (x, A)=\int_{\Omega_{\varepsilon}}\mathbf{1}_{A}(f_{\omega}(x))d\theta_{\varepsilon}(\omega),
 \end{equation}
where $A\in X$ is a measurable set,  $x\in X$ and $\mathbf{1}_{A}$ is the indicator function of the set $A$. A probability measures $\mu_{\varepsilon}$ is called  \textit{stationary} if for any measurable set  $A$ we have:
$$
\mu_{\varepsilon}(A)=\int_{X} L_{\varepsilon}(x,A)d\mu_{\varepsilon}(x)
.$$
We call it an absolutely continuous stationary measure, if it has a density with respect to the Lebesgue measure. \\

We are now ready to introduce the  \em randomly applied stochastic perturbations\em. We first have a Borel  measurable map $f $ acting on the compact set $X;$ then we operate an aleatory reset of the state of the  system at each failure  of a Bernoulli random variable: if $(x_{n})_{n\in\N}$ denotes the successive states of such a random dynamical systems, then at each time $n\in\N$ we have $x_{n+1}=f(x_n)$ with probability $(1-\epsilon)$ and $x_{n+1}=\xi_n$ with probability $\epsilon$, where $\xi_n$ is the realization of a random variable with value in $X$. This kind of perturbation corresponds to the family $(f_\omega)_{\omega\in \Omega_{\varepsilon}}$ of random transformations defined by
\begin{equation}\label{RASP}
f_{\omega}(x)=\eta f(x) + (1-\eta)\xi\qquad\forall\, x\in X,
\end{equation}
where $\omega=(\eta,\xi)$ is a random vector with value in $\Omega_{\varepsilon}=\{0,1\}\times X$. The  two components $\eta$ and $\xi$ of $\omega$ are independent and $\eta$ is a Bernoulli variable with the probability of being $0$ equal to $\eps$, while $\xi$ is a random variable that we will suppose Lebesgue-uniformly distributed on $X$\footnote{In the book of Lasota Mackay \cite{LM} the variable $\xi$ is more generally distributed with a density with respect to the Lebesgue measure; this will not change the results of our paper.}. The joint distribution $\theta_\varepsilon$ of these two components is the product of the Bernoulli measure with weights $(\eps,1- \eps)$ and the uniform measure on $X$.

In order to obtain the stationary measure $\mu_\eps$, let us introduce the random Koopman operator $U_{\eps}:L^{\infty}\to L^{\infty}$ defined for all $\phi\in L^{\infty}$\footnote{From now on $L^1$ and $L^{\infty}$ will be referred to the Lebesgue measure $m$ and the integral with respect to the latter will be denote as $\int (\cdot)\ dx.$} by
\[
U_{\eps}\phi(x):= \int \phi(f_{\omega}(x))d\theta_\varepsilon.
\]
Now, if we take two observables $\phi\in L^{\infty}$ and $\psi\in L^1$, it easy to check that
$$
\int U_{\eps}\phi(x) \psi(x) dx=\int \int \phi( f_{\omega}(x)) \psi(x)d\theta_{\varepsilon} dx=
$$
$$
\eps\int \int\phi(x)\psi(y) dx dy+(1-\eps)\int \phi(f(x))\psi(x)dx\\
= \int \phi(x)P_{\eps}\psi(x) dx,
$$
where $P_\eps$ is the adjoint operator of $U_\eps$, that is the random transfer operator. If we denote $P$  the transfer operator associated to $f$ and $\overline{\psi}= \int \psi(y)dy$, then we have
\begin{equation}\label{PFCONT}
P_{\eps}\psi(x)= (1-\eps)P\psi(x)+\eps \overline{\psi}.
\end{equation}
The stationary measure $\mu_\eps$ verifies $\int \phi(x) d\mu_\eps=\int U_{\eps}\phi(x) d\mu_\eps$ and in our case is  given by $\mu_\eps=h_\eps m$ where $h_\eps\in L^1$ is a density such that $h_\eps=P_{\eps}h_\eps$. Such a density exists and is given by \cite{LM}:
\begin{equation}\label{DENS}
h_\eps=\eps \sum_{k=0}^{\infty}(1-\eps)^kP^k{\bf 1}.
\end{equation}\\

 In the random setting we are developing  the decay of correlations will be of annealed type and it will be formulated in terms of the iterates of the random Koopman operator, in particular we have:

\begin{proposition}\label{Prop:Correlations}
If $\phi\in L^{\infty}$ and $\psi\in L^1\cap L^\infty$, then
\begin{equation}\label{COR}
\left|\int\int  \phi(f^n_{\omg}(x))\psi(x)d\mu_\eps d\theta_\eps^\N-\int \phi(x)d\mu_\eps \int \psi(x)d\mu_\eps\right|\leq 2(1-\eps)^n||\phi||_{L^\infty} ||\psi h_\eps||_{L^1}.
\end{equation}
\end{proposition}
\begin{proof}  To use the duality of the Koopman operator with the transfer operator $P$, let us introduce the following quantity:
\begin{equation}\label{DDCC}
Cor_m(\phi,\psi,n) := \left|\int  U^n_\eps(\phi(x))\psi(x)dx-\int \phi(x)d\mu_\eps \int \psi(x)dx\right|=
\left|\int  \phi(x)P_\eps^n\psi(x)dx-\int \phi(x)d\mu_\eps \int \psi(x)dx\right|.
\end{equation}
By recurrence, it is easy to obtain the following formula for the iterates of $P_{\eps}$:
\begin{equation}\label{ITPF}
P_{\eps}^n\psi=(1-\eps)^n P^n\psi \ + \eps \overline{\psi}
 \sum_{k=0}^{n-1}(1-\eps)^kP^k{\bf 1}.
\end{equation}
By replacing this expression into (\ref{DDCC}) and  writing there $d\mu_{\eps}= h_{\eps}dx$, with $h_{\eps}$ given by (\ref{DENS}), we get
\begin{eqnarray*}
Cor_m(\phi,\psi,n) & =   &\left| (1-\eps)^n\int \phi(x)P^n\psi(x)dx
-\eps\overline{\psi} \int \phi(x)\sum_{k=n}^{\infty}(1-\eps)^k P^k{\bf 1}(x) dx\right|\\
             & \leq & (1-\eps)^n||\phi||_{L^\infty} ||P^n\psi||_{L^1}
+\eps|\overline{\psi}| ||\phi||_{L^\infty} ||\sum_{k=n}^{\infty}(1-\eps)^k P^k{\bf 1} ||_{L^1}\\
& \leq & 2(1-\eps)^n||\phi||_{L^\infty} ||\psi||_{L^1},
\end{eqnarray*}
by H\"older inequality and the fact that the infinite sum  is of positive  terms. Therefore, if $\psi\in L^1\cap L^\infty$, then the correlations \eqref{COR} are equal to
\[
Cor_m(\phi,\psi h_\eps,n)\leq 2(1-\eps)^n||\phi||_{L^\infty} ||\psi h_\eps||_{L^1}.
\]
\end{proof}
\begin{remark}
We will argue later on that in the EVT the preceding decay of correlations result plays a fundamental role whenever
 $\psi$ is a specific characteristic function: in this case it will be sufficient to have the density $h_{\eps}$ in $L^1.$  For observables $\psi$ which are solely in $L^1$ we will need a density in $L^{\infty}.$
 Note that the exponential decay of the correlations does not depend on the map $f$, which is a remarkable property of RASP.
\end{remark}

We are now ready to introduce  the extreme value statistics; at this regard we define a stochastic process $(Y_n)_{n\in\N}$, $n\in\N $,  by composing a given observable $\phi:X\to\R$ as: $Y_n=\phi\circ f^n_{\underline{\omega}},$ with
\[
f_{\omega_n}(x)= \eta_nf(x)+(1-\eta_n)\xi_n,\quad\forall\,x\in X,
\]
where $\omega_n=(\eta_n,\xi_n)\in\{0,1\}\times X$ is random variable which distribution is the product of the Bernoulli measure with weights $(\eps,1- \eps)$ and the uniform measure on $X$.

 In this case  $(Y_n)_{n\in\N}$ will be a stationary process if we consider the probability $\mathbb{P}=\mu_{\varepsilon}\times \theta_{\varepsilon}^{\mathbb{N}}$, which corresponds to the {\em annealed} situation where we average over the initial condition and over the realization of the noise.

 We will in particular choose the observable $\phi(x)=-\log d(x,z)$, where $d(\cdot,\cdot)$ denotes (some)  distance on $X$ and $z$ a point of $X$. As we anticipated in the Introduction, such an observable is related to recurrence in small sets, since the distribution of the maximum of the random observable $\phi \circ f^k_{\omg}, k=0,\dots, n-1$  up to the level $u_n$, coincides with the distribution of the first entrance of the random orbit into the ball (in the metric $d$), $B(z,e^{-u_n})$ centered at $z$ and of radius $e^{-u_n}$

We now state the two conditions which insure weak dependence of the process and which allow to get the limiting distribution of the maxima.\\

\noindent\textbf{Condition}[$D_2(u_n)$]\label{cond:D2} We say that $D_2(u_n)$ holds for the sequence $Y_0,Y_1,\ldots$ if for all $\ell,t$
and $n$,
\[
|\mathbb{P}\left(Y_0>u_n\cap  \max\{Y_{t},\ldots,Y_{t+\ell-1}\leq u_n\}\right)-\mathbb{P}(Y_0>u_n)
  \p(M_{\ell}\leq u_n)|\leq \gamma(n,t),
\]
where $\gamma(n,t)$ is decreasing in $t$ for each $n$, and
$n\gamma(n,t_n)\to0$ when $n\rightarrow\infty$ for some sequence
$t_n=o(n)$.

Now, let $(k_n)_{n\in\N}$ be a sequence of integers such that
\begin{equation}
\label{eq:kn-sequence-1}
k_n\to\infty\quad \mbox{and}\quad  k_n t_n = o(n).
\end{equation}

\noindent\textbf{Condition}[$D'(u_n)$]\label{cond:D'} We say that $D'(u_n)$
holds for the sequence $Y_0, Y_1, Y_2, \ldots$ if there exists a sequence $(k_n)_{n\in\N}$ satisfying \eqref{eq:kn-sequence-1} and such that
\begin{equation}
\label{eq:D'un}
\lim_{n\rightarrow\infty}\,n\sum_{j=1}^{\lfloor n/k_n \rfloor}\mathbb{P}( Y_0>u_n,Y_j>u_n)=0.
\end{equation}
As we said in the Introduction when these two  conditions
 hold for  $(Y_n)_{n\in\N}$ then there exists an extreme value law for $M_n$ and $H(\tau)=1-e^{-\tau},$ but provided  that the limit (\ref{eq:un}) is true.\\


\section{Stationary measure for injective piecewise diffeomorphisms with RASP}
We now introduce the dynamical system upon which we will perform the randomly applied stochastic perturbation.
\begin{definition} {\bf The map.} \label{S1}
Let us   $X$ be  a compact subset of $\mathbb{R}^D$ which is the closure of its interior $X=\overline {\text{int}(X)}$, equipped with the metric $d$ and with the normalized Lebesgue measure $m$; suppose  $\{X_i\}_{i=1}^N$ is a collection of $N$ disjoint open subsets of $X$ such that $X=\bigcup_{i=1}^N {\overline X_i}$ and $m(\Delta)=0$, where $\Delta:=X\setminus\bigcup_{i=1}^N X_i$ is the {\em singular set}. We will consider non singular maps  $f:X\to X$ which are  injective and such that $f|_{X\setminus\Delta}$ is a $C^1$-diffeomorphism. We  say that  $f$ is continuous in a point $w$ of the boundary of $X$ if there is an open ball of radius $\kappa,$ $B(w, \kappa)$ such that $f$ is continuous on $B(w, \kappa)\cap \text{int}(X)$ and it can be extended continuously on $B(w, \kappa)$.
 \end{definition}
 \begin{remark}
 In the examples we will treat, the set $\Delta$ will be composed by points where $f$ is discontinuous and by the points of the boundary of $X$. We would like to stress that a point of  that boundary is not necessarily a  discontinuity point of $f$.
 \end{remark}

The following proposition give the expression of the stationary density.

\begin{proposition}\label{PROPDENSPID} Let $\{\Lambda_{k}\}_{k\in\N}$ be the sequence of sets defined by
$\Lambda_{0}:=X$ and $\Lambda_{k+1}:=f(\Lambda_k\setminus\Delta)$ for all $k\geq 1$.
Denote $J_k(x):=\prod_{l=1}^{k}|\det(f'(f^{-l}(x))|^{-1}$ for all $k\geq 1$ and $J_0(x):={\bf 1}(x)$. Then,
\begin{equation}\label{DENSPID}
h_\eps(x)= \eps\sum_{k=0}^{\infty}(1-\eps)^k J_k(x){\bf 1}_{\Lambda_k}(x)\quad \forall\, x\in X.
\end{equation}
\end{proposition}

Note that $J_k(x)$ is well defined for every $x\in X$ such that $x\notin f^l(\Delta)$ for all $l\in\{1,\dots,k\}$. Among other basics properties of the sets $\Lambda_k$ which will be helpful all along this paper, the following lemma ensures that $J_k(x)$ (and therefore $h_\eps(x)$) is well defined
for all $x\in\Lambda_k$.

\begin{lemma}\label{LEMLAMBDA} The sequence of sets $\{\Lambda_k\}_{k\in\N}$ has the  following properties:
\begin{enumerate}
\item $\Lambda_k$ is open for all $k\geq 1$
\item $\Lambda_{k+1}\subset\Lambda_{k}$ for all $k\in\N$,
\item $\Lambda_k\cap f^p(\Delta)=\emptyset$, for all $p\in\{1,\dots,k\}$.
\end{enumerate}
\end{lemma}


\begin{proof} 1) The set $\Lambda_1$ is open since $X\setminus\Delta$ is open and $f$ is a diffeomorphism in $X\setminus\Delta$.
Using the same arguments, we obtain by induction that $\Lambda_k$ is open for any $k\geq 1$. 2) We have $\Lambda_2=f(\Lambda_1\setminus\Delta)\subset f(X\setminus\Delta)=\Lambda_1$ and if we suppose $\Lambda_{k}\subset\Lambda_{k-1}$, then  $\Lambda_{k+1}=f(\Lambda_k\setminus\Delta)\subset f(\Lambda_{k-1}\setminus\Delta)=\Lambda_k$. 3) By injectivity we have $f(X\setminus\Delta)\cap f(\Delta)=\emptyset$, i.e.
$\Lambda_1\cap f(\Delta)=\emptyset$. Suppose now that $\Lambda_k\cap f^k(\Delta)=\emptyset$ and that there exists $x\in\Lambda_{k+1}\cap f^{k+1}(\Delta)$. Then there exist $y\in\Lambda_k\setminus\Delta$ and $y'\in f^{k}(\Delta)$ such that $f(y)=f(y')=x$. As $f$ is injective, we must have $y=y'$, which contradicts the recurrence hypothesis. It follows that $\Lambda_k\cap f^k(\Delta)=\emptyset$ for all $k\geq 1$ and applying 1) we obtain the desired result.
\end{proof}
\begin{proof}(Proposition \ref{PROPDENSPID}).
We compute  the transfer operator $P$ of $f$ in order to use the formula \eqref{DENS}. Since $f$ is a diffeomorphism on the full Lebesgue measure open set $X\setminus\Delta$, its associated transfer operator acts on a function $\psi\in L^1(X)$ as
\[
P\psi(x)={\bf 1}_{f(X\setminus\Delta)}(x)\psi(f^{-1}(x))|det(f'(f^{-1}(x))|^{-1}={\bf 1}_{\Lambda_1}(x)\psi(f^{-1}(x))J_1(x)\quad\forall\, x\in X.
\]
Since the set $f(\Delta)$ has zero Lebesgue measure, we can suppose without loss of generality that $P\psi(x)=0$ if $x\in f(\Delta)$.
Then, $P\psi(x)$ is well defined for all $x\in X$, since $f(X\setminus\Delta)=\Lambda_1$ does not contain point of $f(\Delta)$. Now, we show by induction that for the uniform density ${\bf 1}={\bf 1}_{X}$ we have
\[
P^k{\bf 1}(x)={\bf 1}_{\Lambda_k}(x)J_k(x) \qquad\forall k\geq 1.
\]
It is easy to check that $P{\bf 1}(x)={\bf 1}_{\Lambda_1}(x)J_1(x)$ and that
\[
P^{k+1}{\bf 1}(x)={\bf 1}_{\Lambda_1}(x){\bf 1}_{f(\Lambda_k)}(x)J_{k+1}(x),
\]
if we suppose the property true at rank $k$. Now, we have
\[
\Lambda_1\cap f(\Lambda_k)=(\Lambda_1\cap \Lambda_{k+1})\cup(\Lambda_1\cap f(\Lambda_k\cap\Delta))=\Lambda_{k+1}\cup(\Lambda_1\cap f(\Lambda_k\cap\Delta)).
\]
Since $\Lambda_1\cap f(\Lambda_k\cap\Delta)\subset \Lambda_1\cap f(\Delta)=\emptyset$, we have $\Lambda_1\cap f(\Lambda_k)=\Lambda_{k+1}$. It follows that $P^{k+1}{\bf 1}(x)={\bf 1}_{\Lambda_{k+1}}(x)J_{k+1}(x)$, which ends the induction. Using \eqref{DENS} we obtain that
\[
h_\eps(x)= \eps\sum_{k=0}^{\infty}(1-\eps)^k J_k(x){\bf 1}_{\Lambda_k}(x)\quad \forall\, x\in X,
\]
which ends the proof of the proposition.
\end{proof}
Now we decompose the space $X$ as the disjoint union $X=\Lambda_1\cup f(\Delta)\cup (X\setminus f(X))$ and we use \eqref{DENSPID} to study the density in each of these sets. Since the  last two sets do not belong to any $\Lambda_k$ with $k\neq 0$, we obtain that $h_\eps=\eps$ on these sets. For $x\in\Lambda_1$, by property 2) of Lemma \ref{LEMLAMBDA}, two situations can occur: either there exists $p\geq 1$ such
that $x\in\Lambda_p\setminus\Lambda_{p+1}$ or $x\in \Lambda:=\bigcap_{k=1}^\infty\Lambda_k$. So in $\Lambda_1$ we have on one hand
\begin{equation}\label{DENSPC}
h_\eps(x)=\eps\sum_{k=0}^{p}(1-\eps)^k J_k(x)\qquad \forall\, x\in \Lambda_{p}\setminus\Lambda_{p+1},
\end{equation}
and the by continuity of the $J_k$, the density is continuous. On the other hand we have
\begin{equation}\label{DENSLAMB}
h_\eps(x)=\eps\sum_{k=0}^{\infty}(1-\eps)^k J_k(x)\qquad\forall\,x\in\Lambda.
\end{equation}
In this last case the density may not be bounded, depending on the behavior of $J_k(x)$. However, if we
suppose that there exists $0<\lambda<(1-\eps)^{-1}$ such that $|det(f'(x)|^{-1}\leq \lambda$ for all $x\in X\setminus\Delta$, then the series \eqref{DENSLAMB} converges uniformly in $\Lambda$ and $h_\eps$ is also continuous in $\Lambda$ and thus on the whole set $\Lambda_1$ (see Proposition \ref{FIN}).

In the next sections we decompose our study of the existence of EVL in two parts. The first part concerns the point of $\Lambda_1\setminus\tilde{\Lambda}$, where
\begin{equation}\label{A}
\tilde{\Lambda}:=\bigcap_{k\in\N}\overline{\Lambda}_k.
\end{equation}
The set $\tilde{\Lambda}$ is called the {\em global attractor} of $f$ and contains all the point of
$\Lambda$. In $\Lambda_1\setminus\tilde{\Lambda}$ we can verify the conditions $D_2(u_n)$ and $D'(u_n)$ to show the existence of EVL, but in $\tilde{\Lambda}$ condition $D'(u_n)$ is not verified.  For this reason in the second part, we introduce other conditions which allow us to show the existence of EVL but with an extremal index different from $1;$ this is achieved by giving  additional property, namely by choosing $f$ as a  piecewise contracting map. It has been shown in \cite{CGMU} that for any piecewise contracting map defined in a compact space, if the global attractor $\tilde{\Lambda}$ does not intersect the set of the discontinuities, then it is composed of a finite number of periodic orbits. Moreover, this condition, for injective maps as our, is generic in $C^0$ topology \cite{CB11}. All these properties will help  to get the existence of EVL in the set $\tilde{\Lambda}$. Let us give a few example to illustrate that
\begin{example}\label{EXBAKER}
 As an easy example of an uncountable hyperbolic attractor, let us consider the baker's map $f$ defined on the unit square $[0,1]\times[0,1]$ iteratively for $n\ge 1$, by

\[
  x_{n+1}= \left \{
   \begin{array}{r c l}
     \gamma_a x_n, \ \ \text{for} \ y_n<\alpha \\
      \frac12+\gamma_bx_n, \ \ \text{for} \ y_n>\alpha\\

   \end{array}
 \right.
\]
\[
 y_{n+1}= \left \{
   \begin{array}{r c l}
     \frac{1}{\alpha} y_n, \ \ \text{for} \ y_n<\alpha \\
      \frac{1}{1-\alpha}(y_n-\alpha)+\gamma_bx_n, \ \ \text{for} \ y_n>\alpha\\

   \end{array}
   \right.
\]
with $0\le x_n,y_n\le 1$, $0<\gamma_a<\gamma_b<\frac12$, $\alpha\le \frac12.$ In this case $\Delta$ is given by the union of the boundary of the square and the segment $y=\alpha$. Moreover $\overline{\Lambda}_1$ is given by the two rectangles $[0, \gamma_a]\times[0,1]$, $[\frac12, \frac12+\gamma_b]\times[0,1].$ Notice that the derivative $f'$ is a diagonal matrix with diagonal elements $\gamma_a$ (or $\gamma_b$), and $\frac{1}{\alpha}$ (or $\frac{1}{1-\alpha}$). We can choose these parameters in order to get $\det(f')<1.$

\end{example}

 Typical examples of PCM for which the attractor is generically a finite set of periodic orbits are piecewise affine \cite{BD09,LU06}. That is, if for each $i\in\{1,\dots,N\}$ the restriction $f_i:=f|_{X_i}$ of the map $f:X\to X$ of Definition \ref{S1} to a piece $X_i$ has the following form:
\[
f_i(x)=A_ix+c_i\qquad\forall\, x\in X_i,
\]
where $A_i:\R^D\to\R^D$ is a linear contraction and $c_i\in X$.

\begin{example} {\bf 1-D case} The simplest example of a piecewise contracting map of this class is given in dimension $1$ in the interval $[0,1]$ by
$f(x)=ax+c$ with $a, c\in(0,1)$ (here $X_1=(0,(1-c)/a)$, $X_2=((1-c)/a,1)$ and $\Delta=\{0, 1, (1-c)/a\}$). For almost all the values of the parameters $a$ and $c$, the global attractor $\widetilde{\Lambda}$ is composed of a unique periodic orbit (which period depends on the specific values of the parameters), but for the remaining set of parameters the global attractor is a Cantor set supporting a minimal dynamics. We refer to \cite{C} for a detailed description of the asymptotic dynamics of this map. We remind that in our paper \cite{FFGV} we investigated in detail the case with $c=0$ and $0$ the (unique) fixed point; we  briefly quote here the simple structure of the density
$$
h_\eps(x)= \eps\sum_{k=0}^{p-1}\frac{(1-\eps)^k}{a^k}\quad \forall\, x\in (a^p, a^{p-1}], \ p\ge 1
$$
 and the formulae for the coefficients of the linear scaling. At this regard
 let $\tau>0$, $y=-\ln(\tau)$ and $u_n=\frac{y}{a_n}+b_n$ with
\[
a_n=1\qand \ b_n= \log\left(2n\eps\sum_{k=0}^{p-1}\frac{(1-\eps)^k}{\alpha^k}\right)\qquad\forall\,n\in\N,
\]
If:\\
(i) $z\neq 0$ on the interval but not in the countably many discontinuity points of $h_\eps$, namely $z\notin \cup_{j\in \mathbb{N}}\{a^j\}$;\\
(ii) $p\geq 1$ such that $z\in (a^p, \alpha^{p-1})$ and $n$  is large enough such that the ball $B(z,e^{-u_n})\subset(a^p, a^{p-1})$, then:
$$
n\mathbb{P}(X_0>u_n) =\tau
 $$
 \end{example}
 \begin{example}{\bf 2-D case}
 In the unit square $X=[0,1]^2$, another interesting example is given by a map with four pieces $X_1:=(0,T_1)\times(0,T_2)$, $X_2:=(T_1,1)\times(0,T_2)$, $X_3:=(0,T_1)\times(T_2,1)$ and $X_4:=(T_1,1)\times(T_2,1)$, singular set $\Delta$ given by the two segments $\{x=T_1\}, \{y=T_2\}$ and the boundary of the unit square,  and restricted maps $f_1(x,y):=a(x,y)+(1-a)(1,0)$, $f_2(x,y):=a(x,y)+(1-a)(1,1)$, $f_3(x,y):=a(x,y)$ and $f_4(x,y):=a(x,y)+(1-a)(0,1)$. Once again, for almost all values of the parameter $a\in[0,1)$, $T_1$ and $T_2$, the global attractor is composed of a finite number of periodic orbits. However, contrarily to the 1 dimensional case, where the number of periodic orbits is bounded by the number of pieces $N$ \cite{NPR}, the number of periodic orbits increases with the parameter $a$ and tends to infinity when $a$ goes to $1$. This example belongs to a larger class of higher dimension piecewise affine contracting maps which are models for genetic regulatory networks and are studied in \cite{CFLM06}.
\end{example}

{\bf Question} An interesting question is the check if the stationary measure is stochastically stable, in the sense that the measure $\mu_{\eps}$ will converge weakly to the SRB measure supported on the global attractor. Let us remind that it is {\em not} the case for the simple contraction $f(x)=ax$, $x\in [0,1]$, $a<1$, as it was proved in Lasota Mackay \cite{LM}.


\section{Convergence to the  extreme value law}
We prove in this section the convergence towards an extreme value distribution in the complement of the attractor defined in (\ref{A}). We will work with the process $Y_n(x,\omg) :=-\log d(f_{\omg}^n(x), z)$. Our first task will be to prove the existence of the limit (\ref{eq:un}); this will be accomplished in two manners: we will first give a general result showing the existence of that limit for {\em any} point of $X$, but without an explicit expression for the scaling coefficients $a_n$ and $b_n$. Such  coefficients could instead be computed on the set $\Lambda_1\setminus\tilde{\Lambda}$ and even on the attractor itself provided the density of the stationary measure will be essentially bounded. The first general result uses Theorem 1.7.13 in the book by Leadbetter, Lindgren and Rootzen \cite{LLR}: according to it, a sufficient condition to guarantee the existence of the limit  (\ref{eq:un}) for $0<\tau<\infty,$ is that $\frac{1-F(x)}{1-F(x-)}\rightarrow 1$, as $x\rightarrow x_{F},$ where $F$ is the distribution function of $Y_0$, the term $F(x-)$ in the denominator denotes the left limit of $F$ at $x$ and finally $x_F:=\sup\{x; F(x)<1\}.$ For our particular observable $Y_0$, $x_F=\infty$ and the probability $\mathbb{P}$ reduces to the stationary measure $\mu_{\eps}$ which is absolutely continuous and therefore non atomic. Therefore $F$ is continuous and the above ratio goes to $1.$ We have thus proved the following proposition:
\begin{proposition}\label{PT}
Suppose $f:X\to X$ satisfies Definition \ref{S1}. Then, for any $z\in X$ defining the observable $Y_0(\cdot)=-\log d(\cdot, z)$ and any $0<\tau<\infty$, there exists a sequence $u_n$ such that
$$
\lim_{n\to\infty}n\mathbb{P}(X_0>u_n) = \tau.
$$
\end{proposition}
As we anticipated above the previous result could be strengthened  whenever $z$ lies in the complement of the global attractor under the general assumption that $h_{\eps}$ is summable. Actually, we will suppose that
$z\in(\Lambda_0\setminus\tilde{\Lambda})\setminus\cup_{k\geq 2}^{\infty}\partial\Lambda_k$, where $\partial\Lambda_k:=\overline{\Lambda}_k\setminus\Lambda_k$; this in turn implies that $z\in\Lambda_p\setminus\overline{\Lambda}_{p+1}\subset \Lambda_p\setminus{\Lambda}_{p+1}$ for some $p\geq 1$.\footnote{We notice that in the case of the baker's transformation of Example \ref{EXBAKER} the boundaries of the $\Lambda_k$ belong to the attractor and therefore we can take simply $z\in (\Lambda_0/\tilde{\Lambda}).$}
We can take advantage of the continuity of the density in such points to obtain a possible scaling sequence
for which the limit (\ref{eq:un}) holds.

\begin{proposition}\label{ANBN} Suppose $f:X\to X$ satisfies Definition \ref{S1}. Let $p\geq 1$ such that $z\in\Lambda_p\setminus\overline{\Lambda}_{p+1}$. Let $\tau>0$, $y=-\ln(\tau)$ and $u_n=\frac{y}{a_n}+b_n$ with $a_n=D$ and $b_n$ such that
\[
m(B(z,e^{-b_n}))= \left(n\eps\sum_{k=0}^{p}(1-\eps)^k J_k(z)\right)^{-1}=\frac{1}{nh_\eps(z)}\qquad\forall\,n\in\N,
\]
where $B(z,r)$ denotes a ball of center $z$ and radius $r$. Then,
\begin{equation}\label{NPX}
\lim_{n\to\infty}n\mathbb{P}(X_0>u_n) = \tau.
\end{equation}
\end{proposition}

\begin{proof} For all $n\in\N$, we have
\[
n\mathbb{P}(X_0>u_n) =  n\int{\bf 1}_{B(z,e^{-u_n})}(x)d\mu_\epsilon=n\int{\bf 1}_{B(z,e^{-u_n})}(x)
h_\eps(x)dx.
\]
Since $(u_n)_{n\in\N}$ is an increasing sequence for any $\tau>0$ and $z$ is in the open set $\Lambda_{p}\setminus\overline{\Lambda}_{p+1}$, if $n$ is large enough $B(z,e^{-u_n})\subset\Lambda_{p}\setminus\Lambda_{p+1}$. Using formula \eqref{DENSPC} and $m(B(z,e^{-u_n}))=\tau m(B(z,e^{-b_n}))$ we obtain:
\begin{equation}
n\tau m(B(z,e^{-b_n}))\eps\sum_{k=0}^{p}(1-\eps)^k\inf_{x\in B(z,e^{-u_n})}J_k(x)\leq
n\mathbb{P}(X_0>u_n)
\leq n\tau m(B(z,e^{-b_n}))\eps\sum_{k=0}^{p}(1-\eps)^k\sup_{x\in B(z,e^{-u_n})}J_k(x).
\end{equation}
Since $J_k$ is continuous in $B(z,e^{-u_n})$ for all $k\in\{0,\dots,p\}$, we have
\[
\lim_{n\to\infty}\inf_{x\in B(z,e^{-u_n})}J_k(x)=J_k(z)=\lim_{n\to\infty}\sup_{x\in B(z,e^{-u_n})}J_k(x).
\]
Using the expression of the sequence $(u_n)_{n\in\N}$, we obtain
\[
\lim_{n\to\infty}n\mathbb{P}(X_0>u_n) =\tau.
\]
\end{proof}

We note that an inspection of the proof of Proposition \ref{ANBN} easily shows that similar results can be obtained for the points of $\Lambda$, but supposing a continuous bounded density. The following proposition shows that the density has such properties in
$\Lambda$ under some condition on the derivative of $f$.

\begin{proposition}\label{FIN}
Let $0<\lambda<(1-\eps)^{-1}$ and suppose that $|\det(f'(x)|^{-1}\le \lambda$, for all $x\in X/\Delta,$ then $h_{\eps}$ is continuous in $\Lambda$ and the same conclusions of Proposition \ref{ANBN} hold.
\end{proposition}
\begin{proof}
 For any $k\ge 0$ and $x\in X$ let $g_k(x) :=\eps(1-\eps)^k J_k(x){\bf 1}_{\Lambda_k}(x)$. Then, $g_k(x) = 0$ if $x\neq\Lambda_k$ and
$g_k(x)\le \eps((1-\eps)\lambda)^k$  if $x\in \Lambda_k$ and therefore
$\sum_{k=0}^p g_k(x)$ converges uniformly on $X$ when $p\rightarrow \infty$, since
$ \sum_{k=0}^{\infty}((1-\eps)\lambda)^k$ converges. Now, let $x_0\in \Lambda$,
then
$$
\lim_{x\rightarrow x_0}h_{\eps}(x)=\lim_{x\rightarrow x_0}\sum_{k=0}^{\infty}g_k(x)=\sum_{k=0}^{\infty}\lim_{x\rightarrow x_0}g_k(x)=\eps\sum_{k=0}^{\infty}(1-\eps)^k\lim_{x\rightarrow x_0}J_k(x)=\eps\sum_{k=0}^{\infty}(1-\eps)^k J_k(x_0)=h_{\eps}(x_0)
$$
We have used that $x_0\in \Lambda_k$ for all $k\ge 0$, which implies on one hand by openness of the $\Lambda_k$ that
$lim_{x\rightarrow x_0} {\bf 1}_{\Lambda_k} (x) = 1$ and in the other hand that $J_k$ is continuous in $x_0$, since for any $x\in \Lambda_k$ we have
$f^l(x)\notin f(\Delta)$  for all $l\le k$ and $f$ is $C^1$ in $X\setminus\Delta.$
\end{proof}
We are now ready to check conditions $D_2(u_n)$ and $D'(u_n)$ which will guarantee the convergence of our process toward an extreme value law, the Gumbel's one in our case. As we said above condition $D_2(u_n)$ is insured by the decay of correlations result established in Proposition \ref{Prop:Correlations}, see the proof of Th. D in Sect. 4.1 in \cite{AFV} for the very general approach. We are thus left with the proof of $D'(u_n)$. We first need:
\begin{lemma} Let $n\in\N$ and $U_n=B(z, e^{-u_n})$ be a  ball centered at a point $z\in X$ with radius $e^{-u_n}$. Let $p\geq 1$ be such that $z\in \Lambda_{p}\setminus\overline{\Lambda}_{p+1}$ then,
\begin{equation}\label{RETURN}
\mathbb{P}((x,\omega_1,\dots,\omega_j): x\in U_n, f^j_{\omg}(x)\in U_n)=\epsilon \mu_\epsilon(U_n)\sum_{k=0}^{\min\{j-1,p\}}(1-\epsilon)^k\int J_k(x){\bf 1}_{U_n}(x) dx \qquad\forall j\geq 1,
\end{equation}
for $n$ large enough. In particular, under the same hypothesis we have,
\begin{equation}\label{RETURNMAJ}
\mathbb{P}((x,\omega_1,\dots,\omega_j): x\in U_n, f^j_{\omg}(x)\in U_n)\leq\mu_\epsilon(U_n)^2 \qquad\forall j\geq 1.
\end{equation}
\end{lemma}

\begin{proof}

Let us denote
\[
\mbox{Pr}_j:=\mathbb{P}((x,\omega_1,\dots,\omega_j): x\in U_n, f^j_{\omg}(x)\in U_n)\qquad\forall\, j\geq 1.
\]
First of all, for all $j\geq 1$ we can write
\[
\mbox{Pr}_j = \int\int{\bf 1}_{U_n}(x){\bf 1}_{U_n}(f^j_{\underline \omega}(x))d\theta_\epsilon^\N d\mu_\epsilon
= \int {\bf 1}_{U_n}(x) U^j_\epsilon ({\bf 1}_{U_n})(x) d\mu_\epsilon
= \int P^j_\epsilon({\bf 1}_{U_n}h_\epsilon)(x){\bf 1}_{U_n}(x) dx.
\]
Using \eqref{ITPF}, we obtain that for all $j\geq 1$
\begin{eqnarray*}
\mbox{Pr}_j&=&\int P^j_\epsilon({\bf 1}_{U_n}h_\epsilon)(x){\bf 1}_{U_n}(x) dx\\
&=& (1-\epsilon)^j\int P^j({\bf 1}_{U_n}h_\epsilon)(x){\bf 1}_{U_n}(x) dx +
\epsilon \mu_\epsilon(U_n)\sum_{k=0}^{j-1}(1-\epsilon)^k\int P^k{\bf 1}(x){\bf 1}_{U_n}(x) dx\\
&=& (1-\epsilon)^j\int {\bf 1}_{U_n}(x){\bf 1}_{U_n}(f^j(x))h_\epsilon(x)dx +
\epsilon \mu_\epsilon(U_n)\sum_{k=0}^{j-1}(1-\epsilon)^k\int J_k(x){\bf 1}_{\Lambda_k}(x){\bf 1}_{U_n}(x) dx\\
&=& (1-\epsilon)^j\mu_\epsilon(U_n\cap f^{-j}(U_n)) + \epsilon \mu_\epsilon(U_n)\sum_{k=0}^{j-1}(1-\epsilon)^k\int J_k(x){\bf 1}_{\Lambda_k\cap U_n}(x) dx.
\end{eqnarray*}
Now suppose $n$ large enough for $U_n$ being a subset of $\Lambda_{p}\setminus\Lambda_{p+1}$. Then one can show that $f^j(U_n)\subset \bigcup_{k=1}^{j} f^k(\Delta)\cup\Lambda_{p+j}$ (using recurrence on $j$). Now, in the union $ \bigcup_{k=1}^{j} f^k(\Delta)$ we have that $\Lambda_p\cap f^k(\Delta)=\emptyset$ for any $k\leq p$. We also have $f^l(\Lambda_p)\cap f^{p+l}(\Delta)=\emptyset$ for any $l\in\N$, since $\Lambda_p\cap f^p(\Delta)=\emptyset$ and $f$ is injective.
Together with $\Lambda_{p+j}\subset\Lambda_p$, this implies that $U_n\cap f^j(U_n)=\emptyset$, and therefore $U_n\cap f^{-j}(U_n)=\emptyset$ for all $j\geq 1$. On the other hand, $U_n\cap \Lambda_k=\emptyset$ for all $k> p$. It follows that
\[
\mbox{Pr}_j=\epsilon \mu_\epsilon(U_n)\sum_{k=0}^{\min\{j-1,p\}}(1-\epsilon)^k\int J_k(x){\bf 1}_{U_n}(x) dx\qquad\forall j\geq 1,
\]
which prove \eqref{RETURN}. Now, it easy to show using \eqref{DENSPC} that
$\mbox{Pr}_j\leq\mbox{Pr}_{p+1}=\mu_\epsilon(U_n)^2$ for any $j\geq 1$.
\end{proof}

\begin{proposition}\label{EVLLN} Let $p\geq 1$ and $z\in \Lambda_{p}\setminus \overline{\Lambda}_{p+1}$, then the maxima of the process $X_n(x,\omg) :=-\ln d(f_{\omg}^n(x),z)$ verifies condition $D'(u_n).$
\end{proposition}

\begin{proof} Suppose $z\in\Lambda_{p}\setminus\overline\Lambda_{p+1}$ for some $p\geq 1$. Consider the sequence $(u_n(y))_{n\in\N}$ of Proposition \ref{ANBN}.
Let $(k_n)_{n\in\N}$ be such that $k_n\to\infty$ when $n\to\infty$ and $k_n=o(n)$. Supposing $n$ large enough and using \eqref{RETURNMAJ}, we deduce that
\[
n\sum_{j=1}^{\left\lfloor\frac{n}{k_n}\right\rfloor}\mathbb{P}((x,\omega_1,\dots,\omega_j): x\in U_n, f_{\underline{\omega}}^j(x)\in U_n)
\leq n\left\lfloor\frac{n}{k_n}\right\rfloor\mu_\epsilon(U_n)^2\leq \frac{(n\mu_\epsilon(U_n))^2}{k_n}.
\]
Using $\lim n\mu_{\eps}(U_n)=\tau$, we obtain that
\[
\lim_{n\to\infty}n\sum_{j=1}^{[\frac{n}{k_n}]}\mathbb{P}((x,\omega_1,\dots,\omega_j): x\in U_n, f_{\omega_j}\circ\cdots \circ f_{\omega_1}(x)\in U_n)=0,
\]
since $n\mu_\epsilon(U_n)\to\tau$ and $k_n\to\infty$ when $n\to\infty$.
In such a case,  $(X_n(x,\underline{\omega}))_{n\in\N}$ verifies conditions $D_2(u_n)$ and $D'(u_n)$ for our sequence $(u_n(y))_{n\in\N}$, wich implies that $\mathbb{P}(M_n\leq u_n(y))$ converges towards the Gumbel law $e^{-e^{-y}}$.
\end{proof}
Collecting all the results of this section we finally proved the following theorem
\begin{theorem}\label{THEVL1}
Suppose $f:X\to X$ satisfies Definition \ref{S1}. Then, for any $z\in(\Lambda_1\setminus\tilde{\Lambda})\setminus\cup_{k\geq 2}^{\infty}\partial\Lambda_k$, the sequence $\{M_n\}_{n\in\N}$ of the maxima of the process defined for every $n\in\N$ by $Y_n(x,\omg) :=-\log d(f_{\omg}^n(x), z)$ admits the Gumbel's law as extreme values distribution.
\end{theorem}



\section{Clustering in the attractor of piecewise contracting maps}

In this section, we complete the results of the results of Theorem \ref{THEVL1} by studying the case when the maximum of the observable is achieved in the set $\tilde{\Lambda}$ of a piecewise contracting map. This creates clustering of exceedances which are responsible for the appearance of an Extremal Index less than $1$. To prove the existence of such an EVL, we follow the techniques developed in \cite{FFT12} that were recently upgraded in \cite{FFT14}. In particular, we will use the notation used in the latter.
Actually in \cite{FFT12} the authors introduced two conditions on the dependence structure of the stochastic processes in order to obtain convergence to an EVL under the presence of clustering. These conditions were called $D^p(u_n)$ and $D'_p(u_n)$ and essentially were adaptations of conditions $D_2(u_n)$ and $D'(u_n)$, specifically designed to cope with clustering. Later, in \cite{FFT14}, the authors introduced some general conditions called $\D(u_n)$ and $\D_q'(u_n)$, which polished and joined all the previous conditions, allowing to address the presence ($q\geq1$) and absence ($q=0$) of clustering at once. We recall this conditions here.

Let us begin to introduce a few notations. For $q\in \mathbb{N}_0$ we put
$$
A_n^{(q)}:=\{Y_0>u_n, Y_1\le u_n,\dots,Y_q\le u_n\}
$$
where again $Y_k,k\in \mathbb{N}_0$ is a $\mathbb{P}$-stationary process, that we will take as that studied in the preceding chapters.
For $s,\ell\in\N$ and an event $B$, let
\begin{equation}
\label{eq:W-def}
\mathscr W_{s,\ell}(B)=\bigcap_{i=s}^{s+\ell-1} T^{-i}(B^c).
\end{equation}

\begin{condition}[$\D(u_n)$]\label{cond:D} We say that $\D(u_n)$ holds for the sequence $Y_0,Y_1,\ldots$ if for every  $\ell,t,n\in\N$ and $q\in\N_0$,
\begin{equation}\label{eq:D1}
\left|\p\left(\A_n\cap
 \mathscr W_{t,\ell}\left(\A_n\right) \right)-\p\left(\A_n\right)
  \p\left(\mathscr W_{0,\ell}\left(\A_n\right)\right)\right|\leq \gamma(q,n,t),
\end{equation}
where $\gamma(q,n,t)$ is decreasing in $t$ for each $q, n$ and, for every $q\in\N_0$, there exists a sequence $(t_n)_{n\in\N}$ such that $t_n=o(n)$ and
$n\gamma(q,n,t_n)\to0$ when $n\rightarrow\infty$.
\end{condition}
For some fixed $q\in\N_0$, consider the sequence $(t_n)_{n\in\N}$, given by condition  $\D(u_n)$ and let $(k_n)_{n\in\N}$ be another sequence of integers such that
\begin{equation}
\label{eq:kn-sequence}
k_n\to\infty\quad \mbox{and}\quad  k_n t_n = o(n).
\end{equation}

\begin{condition}[$\D'_q(u_n)$]\label{cond:D'q} We say that $\D'_q(u_n)$
holds for the sequence $Y_0,Y_1,Y_2,\ldots$ if there exists a sequence $(k_n)_{n\in\N}$ satisfying \eqref{eq:kn-sequence} and such that
\begin{equation}
\label{eq:D'rho-un}
\lim_{n\rightarrow\infty}\,n\sum_{j=q+1}^{\lfloor n/k_n\rfloor-1}\p\left( \A_n\cap T^{-j}\left(\A_n\right)
\right)=0.
\end{equation}
\end{condition}
Now let
\begin{equation}
\label{eq:OBrien-EI}
\vartheta=\lim_{n\to\infty}\vartheta_n=\lim_{n\to\infty}\frac{\p(\A_n)}{\p(U_n)}.
\end{equation}

From \cite[Corollary~2.4]{FFT14}, it follows that if the stochastic process $Y_0, Y_1,\ldots$ satisfies conditions $\D(u_n)$ and $\D'_q(u_n)$ and the limit in \eqref{eq:OBrien-EI} exists then $\lim_{n\to\infty}\p(M_n\leq u_n)= \e^{-\vartheta\tau}$, where the sequence $(u_n)_{n\in\N}$ is such that $\lim_{n\to\infty}n\p(X_0>u_n)=\tau$.

\subsection{Analysis in periodic attractors}

Using the result just mentioned we will show that

\begin{theorem}\label{THEVL1} Suppose that $f:X\to X$ satisfies Definition \ref{S1}. Suppose moreover that $f$ is piecewise contracting, that is, there exists $\alpha\in(0,1)$ such that
for every $i\in\{1,\dots,N\}$ we have
\[
d(f(x),f(y))\leq\alpha d(x,y)\quad\forall\,x,y\in X_i.
\]
Suppose that the set $\tilde{\Lambda}$ does not contain any discontinuity point of $f$, and let $z\in\tilde{\Lambda}$. Consider the random perturbations defined in \eqref{RASP} and let $\p=\mu_\eps\times\theta_\eps^\N$, where $\mu_\eps$ is a stationary probability measure.
Let $Y_0, Y_1, \ldots$ be the stochastic process given by $Y_n(x,\omg) :=-\log d(f_{\omg}^n(x),\mathcal{O}(z))$, where $\mathcal{O}(z)$ is the orbit of $z$. Let $(u_n)_{n\in\N}$ be a sequence such that $\lim_{n\to\infty}n\p(X_0>u_n)=\tau\geq 0$. Then we have
$$
\lim_{n\to\infty}\p(M_n\leq u_n)=\e^{-\eps\tau}.
$$
\end{theorem}

It has been shown in \cite{CGMU} that for any piecewise contracting map defined in a compact space, if the global attractor $\bigcap_{k=1}^\infty\overline{\Lambda}_k$ does not intersect the set of the discontinuities, then it is composed of finite number of periodic orbits. Therefore, under this hypothesis $\mathcal{O}(z)$ is a periodic orbit of the map. If $p$ is the period of this orbit, then $\{Y_0>u_n\}$ occurs if and only if $x\in U_n$, where
\[
U_n=\bigcup_{i=0}^{p-1} B(f^i(z),\e^{-u_n}).
\]

\begin{remark}
Note that on the contrary to the usual choice for the observable function, here, we are considering that the maximum is achieved at every point of the whole periodic orbit, instead of at a single periodic point. In an ongoing work concerning EVLs for observables achieving a maximum on multiple correlated points, the second-named author with D. Azevedo, A.C.M. Freitas and  F.B. Rodrigues are developing a thorough study regarding the consequences of having multiple maxima in terms of the extremal behavior of the system. The important fact to retain here is that by considering the whole orbit we are performing a sort of reduction to the case where the period is 1.
\end{remark}

As in the previous section, Proposition \ref{PT} guaranties the existence of the a sequence $\{u_n\}_{n\in\N}$ such that the limit \eqref{eq:un} holds. However, as a side comment, we note that we can obtain an estimation of the sequence $\{u_n\}_{n\in\N}$ using a similar proof as the one of Proposition \ref{ANBN}, provided the density is bounded and continuous. Since $\tilde{\Lambda}\cap\Delta=\emptyset$, it is composed of finite set of periodic points that are not in $\Delta$ and using the definition of the sets $\Lambda_k$, it is easy to show that all the points of $\tilde{\Lambda}$ also belongs to $\Lambda$, i.e
$\tilde\Lambda=\Lambda$. Thus, we can argue that, under the hypothesis of Theorem \ref{THEVL1} and of Proposition \ref{FIN}, the
density has the desired properties on $\tilde{\Lambda}$ to obtain an estimation of the sequence $\{u_n\}_{n\in\N}$.

The proof of the theorem consists in showing that $Y_0, Y_1,\ldots$ satisfies conditions $\D(u_n)$ and $\D_q'(u_n)$ and moreover $\vartheta$ given in \eqref{eq:OBrien-EI} is well defined and equals $\eps$.
Since $\mathcal{O}(z)\cap\Delta=\emptyset$, for $n$ large enough each ball of the union defining $U_n$ is contained in one of the open pieces $X_i$, and because of the contraction we have that $f(U_n)\subset U_n$.
Therefore, we will consider $q=1$ and use $A_n^{(1)}=\{(x,\omg): x\in U_n, f_{\omg}(x)\notin U_n\}$.

\medskip
\noindent{\bf Verification of $\D(u_n)$:}

\noindent From the proof of Proposition~\ref{Prop:Correlations} we have:
\begin{equation}
\label{eq:correlation}
Cor_m(\phi,\psi,n) := \left|\iint \psi(x)  \phi(f_{\omg}^n(x))d\theta_\eps^\N(\omg)dx- \int \psi(x)dx\int \phi(x)d\mu_\eps(x)\right|\leq 2(1-\eps)^n\|\phi\|_\infty\|\psi\|_1.
\end{equation}
We will use this information about decay of correlations to verify condition $\D(u_n)$, in the case $q=1$. We have
\begin{align*}
\p\left(A_n^{(1)}\cap
 \mathscr W_{t,\ell}(A_n^{(1)})\right)&=\iint \I_{U_n}(x)\I_{U_n^c}(f_{\omega_1})(x))\I_{\mathscr W_{0,\ell}(A_n^{(1)})}(f_{\omg}^t(x))d\theta_\eps^\N(\omg)d\mu_\eps(x)\\
 &=\eps \iiint  \I_{U_n}(x)\I_{U_n^c}(y)\I_{\mathscr W_{0,\ell}(A_n^{(1)})}(f_{\sigma\omg}^{t-1}(y)) dy d\theta_\eps^\N(\omg)d\mu_\eps(x)\\
 &\quad+(1-\eps) \iint  \I_{U_n}(x)\I_{U_n^c}(f(x))\I_{\mathscr W_{0,\ell}(A_n^{(1)})}(f_{\sigma\omg}^{t-1}(f(x))) d\theta_\eps^\N(\omg)d\mu_\eps(x)\\
 &=\eps \mu_\eps(U_n) \iint \I_{U_n^c}(y)\I_{\mathscr W_{0,\ell}(A_n^{(1)})}(f_{\sigma\omg}^{t-1}(y)) d\theta_\eps^{t+\ell-1}(\omega_2,\ldots,\omega_{t+\ell}) dy\\
 &=\eps \mu_\eps(U_n) \iint \I_{U_n^c}(y)\phi(f_{\sigma\omg}^{t-1}(y)) d\theta_\eps^{t-1}(\omega_2,\ldots,\omega_{t}) dy,
\end{align*}
where
$$
\phi(x)=\int \I_{\mathscr W_{0,\ell}(A_n^{(1)})} (x) \theta_\eps^\ell(\tilde\omega_1,\ldots,\tilde \omega_\ell).
$$
In the third equality above we used the fact that if $f(U_n)\subset U_n$ and consequently $\I_{U_n}(x)\cdot\I_{U_n^c}(f(x))=0$ for all $x\in X$. We will use this fact again in the following computation:
\begin{align}
\p\left(A_n^{(1)}\right)&=\iint \I_{U_n}(x)\I_{U_n^c}(f_{\omega_1})(x))d\theta_\eps^\N(\omg)d\mu_\eps(x)
 =\eps \iint  \I_{U_n}(x)\I_{U_n^c}(y) dy d\mu_\eps(x)\nonumber\\
 &\quad+(1-\eps) \int  \I_{U_n}(x)\I_{U_n^c}(f(x)) d\mu_\eps(x)=\eps \mu_\eps(U_n) \int \I_{U_n^c}(y) dy.
 \label{eq:p(An)}
\end{align}
Also
$$
\p\left(W_{0,\ell}(A_n^{(1)})\right)=\iint \I_{\mathscr W_{0,\ell}(A_n^{(1)})} (x) \theta_\eps^\N(\omg)d\mu_\eps(x) = \int \phi(x) d\mu_\eps(x).
$$
It follows by \eqref{eq:correlation}
\begin{align*}
\Bigg|\p\left(A_n^{(1)}\cap
 \mathscr W_{t,\ell}(A_n^{(1)})\right)&-\p\left(A_n^{(1)}\right)\p\left(W_{0,\ell}(A_n^{(1)})\right)\Bigg|\\&
 =\eps\mu_\eps(U_n)\left|
 \iint \I_{U_n^c}(y)\phi(f_{\sigma\omg}^{t-1}(y)) d\theta_\eps^\N(\omg) dy- \int \I_{U_n^c}(y) dy \int \phi(x) d\mu_\eps(x)
  \right|\\
  &= \eps\mu_\eps(U_n)Cor_m(\phi,\I_{U_n^c},t-1)\leq 2\eps\mu_\eps(U_n) (1-\eps)^t
\end{align*}
Hence, we may take $t_n=\log n$, for example, and $\D(u_n)$ is easily verified with $\gamma(n,t)=2\eps\mu_\eps(U_n) (1-\eps)^t$.\\

\medskip
\noindent{\bf Verification of $\D'_q(u_n)$:}

\noindent For $j>1$, let $\mbox{Pr}^{(1)}_j=\p(\{(x,\omg): x\in A_n^{(1)},\; f_{\omg}^j(x)\in A_n^{(1)}\} )$. We have
\begin{align*}
\mbox{Pr}^{(1)}_j&=\iint \I_{U_n}(x)\cdot\I_{U_n^c}(f_{\omg}(x))\cdot\I_{U_n}(f_{\omg}^j(x))\cdot\I_{U_n^c}(f_{\omg}^{j+1}(x))\;d\theta_\eps^\N\; h_\eps(x)\;dx\\
&= \eps \iiint \I_{U_n}(x)\cdot\I_{U_n^c}(y)\cdot\I_{U_n}(f_{\sigma\omg}^{j-1}(y))\cdot\I_{U_n^c}(f_{\sigma\omg}^{j}(y))\;dy\;d\theta_\eps^\N\; h_\eps(x)\;dx\\
&\quad+ (1-\eps)\iint \I_{U_n}(x)\cdot\I_{U_n^c}(f(x))\cdot\I_{U_n}(f_{\sigma\omg}^{j-1}(f(x)))\cdot\I_{U_n^c}(f_{\sigma\omg}^{j}(f(x)))\;d\theta_\eps^\N\; h_\eps(x)\;dx\\
&=\eps\int \I_{U_n}(x)h_\eps(x)\,dx \iint \I_{U_n^c}(y)\cdot\I_{U_n}(f_{\sigma\omg}^{j-1}(y))\cdot\I_{U_n^c}(f_{\sigma\omg}^{j}(y))\;dy\;d\theta_\eps^\N\\
&\leq \eps \mu_\eps(U_n)\iint \I_{U_n^c}(y)\cdot\I_{U_n}(f_{\sigma\omg}^{j-1}(y))\;d\theta_\eps^\N\;dy= \eps \mu_\eps(U_n)\iint \I_{U_n^c}(y)\cdot U_\eps^{j-1} (\I_{U_n})(y)\;dy\\
&= \eps \mu_\eps(U_n)\iint P_\eps^{j-1}(\I_{U_n^c})(y)\cdot \I_{U_n}(y)\;dy,
\end{align*}
where in the third equality we used the fact that if $x\in U_n$ then $f(x)\notin U_n^c$ and consequently $\I_{U_n}(x)\cdot\I_{U_n^c}(f(x))=0$ for all $x\in X$.

Now, using \eqref{ITPF}, we define:
\begin{align*}
\iint P_\eps^{j-1}(\I_{U_n^c})(y)\cdot \I_{U_n}(y)\;dy&=(1-\eps)^{j-1}\!\!\!\int P^{j-1}(\I_{U_n^c})(y)\I_{U_n}(y)dy+\eps m(U_n^c)\sum_{k=0}^{j-2}(1-\eps)^k\!\!\!\int P^k\I(y)\I_{U_n}(y)dy\\
&=(1-\eps)^{j-1}m(U_n^c\cap f^{j-1}(U_n))+\eps m(U_n^c)\sum_{k=0}^{j-2}(1-\eps)^k m(f^{-k}(U_n))\\
&=: I_j+I\!I_j
\end{align*}
Note that to check \eqref{eq:D'rho-un} we need to show that $n\sum_{j=2}^{\lfloor n/k_n\rfloor-1} \mbox{Pr}^{(1)}_j$ vanishes. But $n\sum_{j=2}^{\lfloor n/k_n\rfloor-1} \mbox{Pr}^{(1)}_j\leq n\eps \mu_\eps(U_n) \sum_{j=2}^{\lfloor n/k_n\rfloor-1} I_j+ I\!I_j$. Since $\lim_{n\to\infty}n\mu_\eps(U_n)=\tau>0$, then to verify $\D'_q(u_n)$ we only need to check that both  $\sum_{j=2}^{\lfloor n/k_n\rfloor-1} I_j$ and $\sum_{j=2}^{\lfloor n/k_n\rfloor-1} I\!I_j$ vanish as $n\to\infty$.

Since $\Lambda\cap\Delta=\emptyset$, there exists $k_0:=\min\{k\in\N : \Lambda_k\cap\Delta=\emptyset\}$ and since the sets $\Lambda_k$ are open and satisfy $\Lambda_{k+1}\subset\Lambda_k$ for all $k\in\N$, for a $n$ large enough there is a $k'\geq k_0$ such that $U_n\subset\Lambda_k'$. Therefore, we can define $d_n:=k_0+d'_n$, where $d'_n:=\max\{k\in\N : U_n\subset\Lambda_k\}$. Note that $\lim_{n\to\infty} d_n=\infty$, since the distance $d(\Lambda_k,\Lambda)\to 0$ when $k\to\infty$. Moreover, as for any $k\geq k_0$ we have $\Lambda_{k+1}=f(\Lambda_{k})$, it follows that $f^{-k}(U_n)\subset f^{-k}(\Lambda_{k_0+d_n})=f^{d_n-k}(\Lambda_{k_0})$ for all
$k\in\{0,1,\dots,d_n\}$. Also, for any $k\in\{0,1,\dots,d_n\}$, we have
\[
m(f^{-k}(U_n))\leq m(f^{d_n-k}(\Lambda_{k_0}))\leq\alpha^{-k}\beta_n\quad\text{where}\quad\beta_n:=\alpha^{d_n}m(\Lambda_{k_0}).
\]
Now we can check \eqref{eq:D'rho-un} starting with $\sum_{j=2}^{\lfloor n/k_n\rfloor-1} I_j$ and splitting the sum
$$\sum_{j=2}^{\lfloor n/k_n\rfloor-1} I_j= \sum_{i=1}^{\lfloor n/k_n\rfloor-2}I_{i+1}= \sum_{i=1}^{d_n} I_{i+1}+ \sum_{i=d_n+1}^{\lfloor n/k_n\rfloor-2} I_{i+1}.$$
For the second sum on the right we have:
\begin{align*}
\sum_{i=d_n+1}^{\lfloor n/k_n\rfloor-2} I_{i+1}&\leq\sum_{i=d_n+1}^{\infty} (1-\eps)^i\leq \eps^{-1}(1-\eps)^{d_n+1}\xrightarrow[n\to\infty]{}0,\quad\mbox{because $\lim_{n\to\infty}d_n=\infty$.}
\end{align*}
For the first sum on the right, observe that
\begin{align*}
\sum_{i=1}^{d_n} I_{i+1}&=\sum_{i=1}^{d_n} (1-\eps)^i m(U_n^c\cap f^{-i}(U_n))\leq \sum_{i=1}^{d_n} (1-\eps)^i m(f^{-i}(U_n))\leq \beta_n \sum_{i=1}^{d_n} (1-\eps)^i\alpha^{-i}.
\end{align*}
Now to analyze the asymptotic behavior of $\beta_n \sum_{i=1}^{d_n-1} (1-\eps)^i\alpha^{-i}$ we
 consider three different cases. We assume first that $(1-\eps)\alpha^{-1}>1$.

\begin{align*}
\beta_n\sum_{i=1}^{d_n}(1-\eps)^i\alpha^{-i}&= \beta_n\frac{[(1-\eps)\alpha^{-1}]^{d_n+1}-(1-\eps)\alpha^{-1}}{(1-\eps)\alpha^{-1}-1}\leq\alpha^{d_n} m(\Lambda_{k_0})\frac {[(1-\eps)\alpha^{-1}]^{d_{n+1}}}{(1-\eps)\alpha^{-1}-1}\\
&\leq \frac{m(\Lambda_{k_0})}{(1-\eps)-\alpha}(1-\eps)^{d_{n+1}}\xrightarrow[n\to\infty]{}0, \mbox{ since $\lim_{n\to\infty}d_n=\infty$.}
\end{align*}
In the case $(1-\eps)\alpha^{-1}<1$, then
$$
\beta_n \sum_{i=1}^{d_n} (1-\eps)^i\alpha^{-i}\leq \beta_n \sum_{i=1}^{\infty} (1-\eps)^i\alpha^{-i}\xrightarrow[n\to\infty]{}0 , \mbox{ since $\lim_{n\to\infty}\beta_n=0$.}
$$
In the case $(1-\eps)\alpha^{-1}=1$, then
$$
\beta_n \sum_{i=1}^{d_n} (1-\eps)^i\alpha^{-i}\leq \beta_n d_n \leq d_n\alpha^{d_n}\xrightarrow[n\to\infty]{}0, \mbox{ since $\lim_{n\to\infty}d_n=\infty$.}
$$

Now, we turn to $\sum_{j=2}^{\lfloor n/k_n\rfloor-1} I\!I_j$. Let $\beta_n$ and $d_n$ be defined as above, then

\begin{align*}
\sum_{j=2}^{\lfloor n/k_n\rfloor-1}I\!I_j&=\sum_{i=0}^{\lfloor n/k_n\rfloor-3}I\!I_{i+2}
=\sum_{i=0}^{d_n} \eps m(U_n^c)\sum_{k=0}^{i}(1-\eps)^k m(f^{-k}(U_n))+\sum_{i=d_n+1}^{\lfloor n/k_n\rfloor-3}\eps m(U_n^c)\sum_{k=0}^{i}(1-\eps)^k m(f^{-k}(U_n))\\
&\leq\sum_{i=0}^{d_n}\eps\sum_{k=0}^{i}(1-\eps)^k \alpha^{-k}\beta_n+\sum_{i=1}^{\lfloor n/k_n\rfloor}\eps \sum_{k=0}^{\infty}(1-\eps)^k m(f^{-k}(U_n))\\
&\leq\eps\beta_n\left(1+\sum_{i=1}^{d_n}\sum_{k=0}^{i}(1-\eps)^k \alpha^{-k}\right)+\frac{n}{k_n}\mu_\eps(U_n)\\
\end{align*}
Here we have used \eqref{DENS} to obtain
\[
\mu_\eps(U_n)=\int\I_{U_n}(x)h_\eps(x)dx=\eps\sum_{k=0}^{\infty}(1-\eps)^k\int P^k\I(x)\I_{U_n}(x)dx=\eps\sum_{k=0}^{\infty}(1-\eps)^k m(f^{-k}(U_n)).
\]

The second term on the right tends to $0$ when $n$ goes to infinity, since $n\mu_\eps(U_n)\to\tau$ and $k_n\to\infty$ when $n$ goes to $\infty$. For the first term we assume first that $(1-\eps)\neq \alpha$. Then,
\begin{align*}
\eps\beta_n\left(1+\sum_{i=1}^{d_n}\sum_{k=0}^{i}(1-\eps)^k \alpha^{-k}\right)&\leq  \eps\beta_n+\frac{\eps\alpha}{(1-\eps)-\alpha}\beta_n\sum_{i=1}^{d_n}[(1-\eps)\alpha^{-1}]^{i+1}\\
&= \eps\beta_n+\frac{\eps(1-\eps)}{(1-\eps)-\alpha}\beta_n\sum_{i=1}^{d_n}(1-\eps)^{i}\alpha^{-i}
\end{align*}
Since, as we have seen above, $\lim_{n\to\infty}  \beta_n \sum_{i=1}^{d_n}(1-\eps)^i\alpha^{-i}=0$, then $\lim_{n\to\infty}\sum_{j=2}^{\lfloor n/k_n\rfloor-1}I\!I_j=0$. Now, if $(1-\eps)\neq \alpha$, we have
\begin{align*}
\eps\beta_n\left(1+\sum_{i=1}^{d_n}\sum_{k=0}^{i}(1-\eps)^k \alpha^{-k}\right)&=\eps m(\Lambda_{k_0})\left(\alpha^{d_n}+d_n\alpha^{d_n}+\frac{1}{2}d_n^2\alpha^{d_n}\right)\xrightarrow[n\to\infty]{}0.
\end{align*}

\medskip
\noindent{\bf Convergence of $\vartheta_n$:}

\noindent As in equation \eqref{eq:OBrien-EI} we consider $\vartheta_n=\frac{\p(A_n^{(1)})}{\p(U_n)}$. By \eqref{eq:p(An)} we have
$$
\vartheta_n=\frac{\p(A_n^{(1)})}{\p(U_n)}=\frac{\eps\mu_\eps(U_n)m(U_n^c)}{\mu_\eps(U_n)}=\eps m(U_n^c)\xrightarrow[n\to\infty]{}\eps=:\vartheta.
$$
This means that the Extremal Index $\vartheta$ is given by the noise level $\eps$.

\medskip\noindent\textbf{Acknowledgement}. JMF was partially supported by FCT grant SFRH/BPD/66040/2009 (financed by the program POPH/FSE) and also by CMUP, which is financed by FCT (Portugal) through the programs POCTI and POSI, with national and European structural funds, under the project PEst-C/MAT/UI0144/2013. JMF and SV are also supported by FCT (Portugal) project PTDC/MAT/120346/2010, which is financed by national and European structural funds through the programs  FEDER and COMPETE. PG was supported by CONICYT project Anillo ACT1112. SV was supported by the ANR- Project {\em Perturbations}, by the project {\em Atracci\'on de Capital Humano Avanzado del Extranjero} MEC 80130047, CONICYT, at the CIMFAV, University of Valparaiso, by the projet {\em MODE TER COM} supported by Region PACA, France and by the project {\em BREUDS}, Brazilian-European partnership in
Dynamical Systems.


\begin{thebibliography}{100}



\bibitem{AFV} H. Aytac, J. M. Freitas, and S. Vaienti. Laws of rare events for deterministic and random dynamical systems. {\it Trans. Amer. Math.}, Published online with DOI: 10.1090/S0002-9947-2014-06300-9, 2014.

\bibitem{BD09}
Bruin H and Deane J
2009
Piecewise contractions are asymptotically periodic
{\it Proceedings of the American Mathematical Society} {\bf 137} (4) 1389-1395

\bibitem{CB11} Catsigeras E and Budelli R 2011 Topological dynamics of generic piecewise contracting maps in $n$ dimensions {\it International Journal of Pure and Applied Mathematics} {\bf 68} (1) 61-83

\bibitem{CGMU} E. Catsigeras, P. Guiraud, A. Meyroneinc and E. Ugalde, On the asymptotic properties of piecewise contracting maps, arXiv:1108.1501.

\bibitem{C} R. Coutinho, Din\^amica simb\'olica linear. PhD Thesis, Technical University of Lisbon, 1999.

\bibitem{CFLM06}
Coutinho R, Fernandez B, Lima R and Meyroneinc A
2006
Discrete time piecewise affine models of genetic regulatory networks
{\it Journal of Mathematical Biology} {\bf 52} (4) 524-570

\bibitem{FFGV}D. Faranda, J.M. Freitas, P. Guiraud, S. Vaienti,  Sampling local properties of attractors via extreme value theory, to appear on {\em Chaos, Solitons and Fractals}, http://arxiv.org/pdf/1407.0412.pdf

\bibitem{FFLTV} D. Faranda, J. Freitas, V. Lucarini, G. Turchetti, S. Vaienti, Extreme Value Statistics for Dynamical Systems with Noise, {\em Nonlinearity}, {\bf 26}, 2597-2622, (2013)

\bibitem{FV1} D. Faranda, S. Vaienti, Extreme Value laws for dynamical systems under  observational noise, {\em Physica D}, \textbf{280}: 86-94, (2014).

\bibitem{FV2} D. Faranda, S. Vaienti, A new recurrences based technique for detecting robust extrema in long temperature records, {\em Geoph. Resear. Lett.} \textbf{40.21}: 5782-5786 (2013).


\bibitem {FF} A. C. M. Freitas and J. M. Freitas, On the link between dependence and independence in extreme
value theory for dynamical systems, {\em Statist. Probab. Lett.} {\bf 78} (2008), no. 9, 1088-1093.

\bibitem {FFT} A. C. M. Freitas, J. M. Freitas, and M. Todd, Extreme value laws in dynamical systems for
non-smooth observations, {\em J. Stat. Phys.},  {\bf 142} (2011), 108-126.

\bibitem {FFT12} A. C. M. Freitas, J. M. Freitas, and M. Todd, The extremal index, hitting time statistics and periodicity, {\em Adv. Math.},  {\bf 231} (2012), 2626-2665.

\bibitem{FFT14} A. C. M. Freitas, J. M. Freitas, and M. Todd. Speed of convergence for laws of rare events and escape rates. {\it Stochastic Process. Appl.}, Published online with DOI: 10.1016/j.spa.2014.11.011, 2014.






\bibitem{LLR} M. R. Leadbetter, G. Lindgren, and H. Rootzen, {\em Extremes and related properties of random
sequences and processes}, Springer Series in Statistics, New York: Springer-Verlag (1983).


\bibitem{LM} Lasota, A., and Mackey, M. C. (1994). Chaos, fractals, and noise: stochastic aspects of dynamics (Vol. 97). Springer.

\bibitem{LU06}
Lima R and Ugalde E
2006
Dynamical complexity of discrete-time regulatory networks
{\it Nonlinearity} {\bf 19} (1) 237-259

\bibitem{NPR} A. Nogueira, B. Pires and A. Rafael, Rosales Piecewise contractions defined by iterated function systems, arXiv:1408.1663



\end{thebibliography}
\end{document}